\documentclass{article}
\title{Pairwise optimal coupling of multiple random variables}
\author{Omer Angel
  \and Yinon Spinka}
\date{May 2021}

\usepackage{amsmath,amssymb,amsfonts,amsthm}

\usepackage[colorlinks=true,citecolor=black,urlcolor=blue,pdfborder={0 0 0}]{hyperref}

\hypersetup{linkcolor=[rgb]{0,0,0.6}}

\usepackage{graphicx}
\usepackage{enumitem}
\usepackage{color}

\usepackage[font=sf, labelfont={sf,bf}, margin=1cm]{caption}

\usepackage[protrusion=true]{microtype}
\usepackage[margin=3cm]{geometry}

\usepackage[nameinlink]{cleveref}
  \crefname{theorem}{Theorem}{Theorems}
  \crefname{thm}{Theorem}{Theorems}
  \crefname{mainthm}{Theorem}{Theorems}
  \crefname{lemma}{Lemma}{Lemmas}
  \crefname{lem}{Lemma}{Lemmas}
  \crefname{remark}{Remark}{Remarks}
  \crefname{prop}{Proposition}{Propositions}
  \crefname{defn}{Definition}{Definitions}
  \crefname{corollary}{Corollary}{Corollaries}
  \crefname{section}{Section}{Sections}
  \crefname{figure}{Figure}{Figures}

\newtheorem{thm}{Theorem}

\newtheorem{lemma}[thm]{Lemma}
\newtheorem{corollary}[thm]{Corollary}
\newtheorem{prop}[thm]{Proposition}
\theoremstyle{definition}

\newtheorem*{remark*}{Remark}

\newtheorem{question}[thm]{Question}

\renewcommand{\P}{\mathbb P}

\newcommand{\E}{\mathbb E}
\newcommand{\R}{\mathbb R}

\newcommand{\eps}{\varepsilon}

\newcommand{\dTV}{d_{\mathrm TV}}
\newcommand{\cS}{\mathcal S}
\newcommand{\cA}{\mathcal A}
\newcommand{\II}{\rm{I\!I}}
\DeclareMathOperator{\Leb}{Leb}
\DeclareMathOperator{\Exp}{Exp}
\DeclareMathOperator*{\argmin}{argmin}

\AtEndDocument{
  \bigskip
  \small
  \textsc{Omer Angel, Yinon Spinka} \par
  \textsc{Department of Mathematics, University of British Columbia} \par
  \textit{Email:} \texttt{\{angel,yinon\}@math.ubc.ca} \par
}

\begin{document}

\maketitle
\begin{abstract}
  We generalize the optimal coupling theorem to multiple random variables:
  Given a collection of random variables, it is possible to couple all of them so that any two differ with probability comparable to the total-variation distance between them.
  In a number of cases we show that the disagreement probability we achieve is the best possible. The proofs of sharpness rely on new results in extremal combinatorics, which may be of independent interest.
\end{abstract}

\section{Introduction}

A \textbf{coupling} of a collection of random variables $(X_i)_{i\in I}$ is a set of variables $(X'_i)_{i\in I}$ on some common probability space with the given marginals, i.e.\ $X_i$ and $X'_i$ have the same law.
We omit the primes when there is no risk of confusion.
Thus, we think of a coupling as a construction of random variables $(X_i)_{i\in I}$ with prescribed laws.

The \textbf{total variation distance} between two random variables $X$ and $Y$ is defined as
\[
  \dTV(X,Y) = \sup_A \big\{|\P(X\in A) - \P(Y\in A)|\big\},
\]
where the supremum is over all (measurable) sets $A$.
The fundamental, classical theorem relating the total variation distance to coupling is the following.

\begin{thm}
  \label{T:couple2}
  For any two random variables $X$ and $Y$, there exists a coupling such that $\P(X\neq Y) = \dTV(X,Y)$.
  Moreover, for any coupling, $\P(X\neq Y) \ge \dTV(X,Y)$.
\end{thm}

As remarked, technically the coupling is a construction of random variables $X'$ and $Y'$ on some probability space with measure $\P'$ so that $X$ and $X'$ have the same law, and similarly $Y$ and $Y'$.
However, following common practice in probability theory, we do not stress the distinction between $X$ and $X'$.
Thus, we use $\P$ for the new probability measure and $X$ and $Y$ for the new variables.
This is a slight abuse of notation which should not cause any difficulty.

\cref{T:couple2} is very simple, and could even be called folklore.
See e.g. \cite{TE} for another recent application of this coupling (under the name Poisson functional representation).
According to Lindvall's overview of Doeblin's life and work \cite{Lind1}, couplings and the inequality in \cref{T:couple2} originated in Doeblin's work in the 30's.
Since that time, coupling has become an important tool in probability theory with numerous applications.
We refer the reader to \cite{fdH,Lind,thorisson2000coupling} for a partial review of applications of couplings.

\medskip

The starting point for the present work is the following observation, which while basic, is not well known:
When coupling more than two random variables, the total variation bound cannot in general be achieved simultaneously for all pairs.
(While the term \emph{coupling} hints at having two random variables, it is standard practice to use it also for larger collections.)	
For example, let $X\in\{0,1\}$, $Y\in\{0,2\}$ and $Z\in\{1,2\}$ each be uniform on the two possible values.
Then $\dTV(X,Y)=\dTV(X,Z)=\dTV(Y,Z)=\frac 12$.
However, in any coupling of $\{X,Y,Z\}$, at least two of the three pairs are unequal.
Thus,
\[ \P(X\neq Y) + \P(X\neq Z) + \P(Y\neq Z) \geq 2 ,\]
and the disagreement probabilities are not all equal to $\frac 12$.

The following result is a generalization of \cref{T:couple2} with a slightly higher probability of disagreement.
The main objective of this paper is to describe and study two constructions that imply this theorem, and to investigate its optimality.
Indeed, in certain cases we show that the given bound is best possible.

\begin{thm}\label{T:main}
  Let $\cS$ be any collection of random variables,
  all absolutely continuous w.r.t.\ a common $\sigma$-finite measure.
  Then there exists a coupling of the variables in $\cS$ such that, for any $X,Y \in \cS$,
  \[
    \P(X \neq Y) \leq \frac{2\dTV(X,Y)}{1+\dTV(X,Y)} .
  \]
\end{thm}

Let us highlight three special cases of this result.
If the reference measure, $\mu$, is the Lebesgue measure on $\R$, then \cref{T:main} yields a coupling of all continuous real random variables.
A second case is when $\mu$ is the counting measure on some countable set $\Omega$, then we get a coupling of all variables taking values in $\Omega$.
Finally, if $\cS$ is a countable collection of random variables, it is always possible to find a measure $\mu$ such that all are continuous w.r.t.\ $\mu$ (indeed, take any non-trivial mixture of their laws).

Somewhat curiously, there are two fairly different constructions of couplings, both of which realize the bound in the theorem, which we describe in \cref{sec:couplings}.
One construction is more naturally adapted to continuous random variables and the other to discrete, though either can be used to prove \cref{T:main}.
While both constructions achieve lower disagreement probabilities in some cases, the worst-case disagreement probability is the same in both.
The two constructions are described in \cref{sec:couplings} and \cref{T:main} is deduced from their analysis.

\medskip

Some forms of this theorem have appeared in the past, and the constructions we describe below can also be viewed as generalizations of previously used methods.
We have not found in the literature any detailed proof of this result.
Since the proof (by either of our constructions) is very short, we include it below.
The best reference we are aware of is by Barak et al. \cite[Lemma 4.1]{Barak}, which reads almost identical to \cref{T:main}, except that the inequality is replaced by equality, and that the family of random variables is (implicitly) assumed to be finite.
(\cite{Barak} is an extended abstract without a detailed proof, and we were unable to locate the full version of that paper.)
The basic idea used there is attributed to Broder \cite{broder}.
Broder was interested in algorithmically measuring similarity between documents, and used the observation that if elements of $A\cup B$ are ordered by a uniform permutation $\pi$, and $h_\pi(S)$ is the $\pi$-minimal element of $S$ then $\P(h_\pi(A)=h_\pi(B)) = \frac{|A\cap B|}{|A\cup B|}$.
For the random variables $X,Y$ that are respectively uniform on finite sets $A$ and $B$, and in the special case that $|A|=|B|$, this equals $1-\frac{2\dTV(X,Y)}{1+\dTV(X,Y)}$.
This can be seen as a special case of Coupling \II\ below.

A different approach was used by Kleinberg and Tardos \cite{KT} for rounding fractional solutions of linear programming problems to integer solutions.
Their approach is to apply von-Neumann's rejection sampling in a discrete setting.
Lemma 3.2 of \cite{KT} gives \cref{T:main} in the case of variables taking values in a common finite set, with the slightly worse bound $2\dTV$ instead of $\frac{2\dTV}{1+\dTV}$.
In that case, the Kleinberg--Tardos approach can be seen as a special case of Coupling I below.
Charikar \cite{Charikar} has connected these two approaches, showing how the Kleinberg--Tardos rounding algorithm can be seen as a generalization of Broder's idea, and that it can be used for approximating several similarity measures of distributions.
We remark that while Coupling \II\ is also a generalization of Broder's minimal element procedure to non-uniform distributions, it is genuinely different from the Kleinberg--Tardos one.
The difference is demonstrated by \cref{fig:triangles}.

\medskip


The fact that the bound of \cref{T:main} comes up in different constructions raises the possibility that it is optimal.
For a function $f \colon [0,1]\to[0,1]$, let us say that $f$ is a \textbf{disagreement bound} if for any finite collection of random variables there is a coupling of the variables so that any two of them, say $X$ and~$Y$, satisfy
\[ \P(X\neq Y) \leq f(\dTV(X,Y)) .
\]
Note that by taking limits it follows that the same bound on disagreement probabilities can be achieved for countable families of random variables.
Then \cref{T:main} states that
\[ F(x) := \frac{2x}{1+x} \]
is a disagreement bound.
It is natural to ask whether there are any smaller disagreement bounds.
The trivial lower bound (see \cref{T:couple2}) is that any disagreement bound must have $f(x)\geq x$ for all $x$.
The example presented before \cref{T:main}, of three variables each taking two possible values, shows that any disagreement bound must have $f(\frac 12)\geq \frac 23 = F(\frac 12)$.
More generally, we show that any disagreement bound must have $f(x) \ge F(x)$ for $x=\frac1n$ and for $x=1-\frac1n$ for all positive integers $n$, as well as for some other rational numbers (see \cref{prop:1/n-bound,prop:k/n-bound,prop:near1}).
We do not know whether such a pointwise bound holds at every point $x \in (0,1)$.
Nevertheless, we provide a lower bound at any point $x$, which improves on the trivial lower bound $f(x)\geq x$, and is asymptotic to $F(x)$ as $x \to 0$ (see \cref{prop:all-x-bound}).
Some of these bounds are depicted in \cref{fig:bounds}.
Moreover, we show that $F$ is optimal, in the sense that no disagreement bound can simultaneously improve on $F$ everywhere, or even on an open interval:

\begin{thm}\label{T:optimal-bound}
  If a disagreement bound is pointwise smaller-or-equal than $F$, then it coincides with~$F$. Moreover, if a disagreement bound is pointwise smaller-or-equal than $F$ on some open interval, then it coincides with~$F$ on that interval.
\end{thm}

\begin{corollary}[\cite{optimal}]\label{cor:optimal}
  Every non-decreasing disagreement bound is pointwise larger-or-equal than $F$.
\end{corollary}

Optimality of the bound $2x/(1+x)$ arising in the Broder and Kleinberg--Tardos constructions has been investigated in \cite{optimal}.
Their model is somewhat different from the coupling one.
There, Alice and Bob are required to sample from two distributions, each known only to one of them, using access to shared randomness but with no communication.
Their goal is to maximize the probability of selecting the same value.
It is not assumed that Alice and Bob use the same strategy.
However, if one requires the strategies to be identical --- as is the case in prior constructions --- then the strategy naturally provides a coupling of more than two distributions.
The definition in \cite{optimal} differs from ours in another key aspect: there it is required (in our notations) that if $\dTV(X,Y)\leq x$ then $\P(X\neq Y)\leq f(x)$.
This is logically equivalent to restricting attention to \emph{non-decreasing} disagreement bounds,
and in that context they prove \cref{cor:optimal}.
Whether or not the monotonicity assumption of \cref{cor:optimal} can be removed remains an open problem.

Optimality of $F$ as a disagreement bound, in both the local and global senses is discussed in \cref{sec:optimality}.

\paragraph{Relation to multi-marginal optimal transport.}

Optimal transport gives rise to a theory analogous to couplings, with many parallels.
For example, Kantorovich's duality theorem is the equivalent to \cref{T:couple2}.
The question of optimal couplings of multiple random variables is closely related to the problem of multi-marginal optimal transport.
The terminology used in that context is different from the probabilistic terminology that we use.
In multi-marginal optimal transport, one is given a \emph{cost function} $\phi \colon \R^d\times\R^d \to [0,\infty)$ and probability measures $\mu_1,\dots,\mu_n$ on $\R^d$. 
Most commonly, one studies convex cost functions such as $\phi(x,y)=\|x-y\|_p^q$. 
For total variation distances, the relevant cost function is $\phi(x,y) = \mathbf{1}_{x\neq y}$.
(Even more generally, there would be a cost function on $n$-tuples $\phi \colon (\R^d)^n\to[0,\infty)$, though the case of a pairwise cost is already of interest.)

A \emph{plan} is a probability measure $\mu$ on $(\R^d)^n$ whose projections are the given $\mu_1,\dots,\mu_n$.
If $\mu_i$ is taken to be the law of a random variable $X_i$, then a plan is nothing other than a coupling of the random variables.
The objective is to determine the infimum $\inf_\mu \sum_{i,j} \int \phi(x_i,x_j) d\mu$, and find optimal $\mu$.
This value is clearly at least $\sum_{i,j} \inf_\mu \int \phi(x_i,x_j) d\mu$.
In probabilistic terms, we let $d_\phi(X_i,X_j) := \inf \E_\mu \phi(X_i,X_j)$, where the infimum is over all couplings, so that the statement is
\[
\inf_\mu \sum_{i,j} \E_\mu \phi(X_i,X_j) \geq 
\sum_{i,j} d_\phi(X_i,X_j).
\]
It is natural to ask how far apart the two quantities above can be.
It is a simple observation that 
\begin{equation}
  \label{eq:MMOT}
  \inf_\mu \sum_{i,j} \E_\mu \phi(X_i,X_j) \leq 
  2c \sum_{i,j} d_\phi(X_i,X_j),
\end{equation}
where $c$ is any constant such that $\phi(x,z) \le c(\phi(x,y)+\phi(y,z))$ for all $x,y,z \in \R^d$. 
Indeed, if one uniformly picks $k\in\{1,\dots,n\}$ and uses the optimal pairwise coupling of each $X_i$ with $X_k$, one gets the bound \eqref{eq:MMOT}.
The main difference between the multi-marginal optimal transport problem and the one we consider is that we aim to get a good upper bound on $\E_\mu \phi(X_i,X_j)$ for every $i$ and $j$, and not merely on the sum.
We refer the reader to \cite{Pass} for an introduction to multi-marginal optimal transport.

\section{Coupling constructions}
\label{sec:couplings}

In this section, we prove \cref{T:main}.
We give two different constructions of couplings, each of which leads to a proof of \cref{T:main}.
We write $a \wedge b$ and $a \vee b$ for the minimum and maximum of $a$ and $b$, respectively.

\subsection{Coupling I}

Our first construction of a coupling is especially suited for continuous  random variables, i.e., which have a density function.
We say that a random variable $X$ is continuous with respect to a measure $\mu$ if there is a density function $g$ such that $\P(X\in A)=\int_A g \,d\mu$.
Note that we do not require $\mu$ to be the Lebesgue measure.
Thus, if $\mu$ is the counting measure on a countable set, then $X$ is continuous with respect to $\mu$ if it is discrete and supported in that set.

\begin{prop}\label{P:couplingI}
  Let $(\Omega,\mu)$ be a $\sigma$-finite measure space.
  For any collection $\cS$ of random variables, all continuous with respect to $\mu$, there exists a coupling such that, for any $X,Y \in \cS$ with densities $g,h$,
  \begin{equation}
    \label{eq:coup1_pr}
    \P(X \neq Y) = F(\dTV(X,Y)) - \frac{1}{1+\dTV(X,Y)} \int_{\Omega} (g \wedge h) \cdot |\P(X=x)-\P(Y=x)| \, d\mu(x) .
  \end{equation}
  In particular, if $\mu$ has no atoms, then
  $\P(X \neq Y) = F(\dTV(X,Y))$.
\end{prop}

\cref{T:main} is a direct consequence of \cref{P:couplingI}.
The coupling is based on a folklore construction of a random variable in terms of a Poisson point process, which is a form of von-Neumann rejection sampling.
As noted above, in the case of distributions on a finite set, Coupling I simplifies to the Kleinberg--Tardos rounding scheme.

\begin{proof}
  Let $\cS=\{X_i\}_i$ and let the density of $X_i$ be $f_i$. 
  We begin with a Poisson point process on $\Omega\times\R_+^2$.
  Specifically, let $\cA$ be a Poisson point process with intensity $\mu\times\Leb\times\Leb$ on $\Omega\times\R_+\times\R_+$, where $\Leb$ is the Lebesgue measure on $\R_+$.
  We denote the points of $\cA$ as $(x,s,t)$, and think of the third coordinate as a time coordinate.
  Given the set $\cA$, define $\cA_i := \{(x,s,t) \in \cA : s\leq f_i(x)\}$.
  We define the random variables by $X_i=x$ if $(x,s,t)\in\cA_i$ has the minimal $t$ among all points of $\cA_i$.  
  If $\cA_i$ does not have a unique point with minimal $t$, we assign $X_i$ an arbitrary value. 
  This happens if $\cA_i$ is empty, or has multiple points with equal minimal $t$, or has no point with $t$-coordinate equal to the infimum of all $t$-coordinates. 
  All of these have probability $0$, so the value of $X_i$ on these events does not affect its law or the disagreement probabilities.

  To see that $X_i$ has the required law (so that the above is indeed a coupling), think of points $(x,s)$ appearing at rate $1$ in time, and intensity $\mu\times\Leb$ on the half plane.
  Points with $s > f_i(x)$ are ignored. 
  Points with $s\leq f_i(x)$ appear at total rate $1$, so there is almost surely a first such point.
  The probability that the $x$-coordinate of the first such point is in some set $A$ is $\int_A f_i(x) dx = \P(X_i\in A)$, as required.

  \medskip

\begin{figure}
 \centering
 \includegraphics[scale=0.7,trim={2cm 0cm 1cm 0cm},clip]{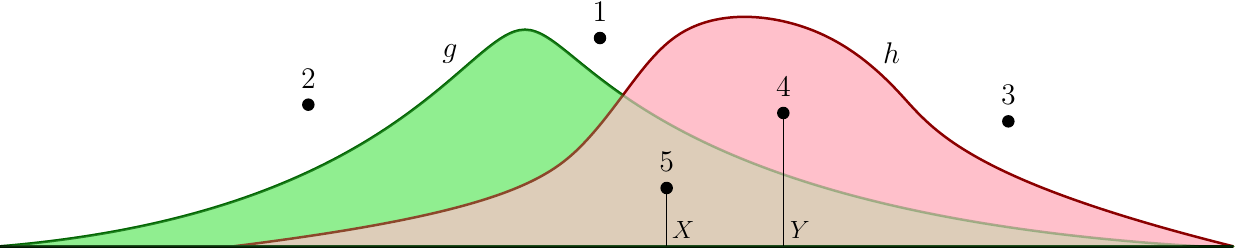}
 \caption{Illustration of Coupling \rm{I}.
   The densities $g$ and $h$ of $X$ and $Y$ are depicted.
   Points arrive according to a Poisson point process.
   Five such points are depicted, with the numbers indicating their relative order of arrival.
   The first point to fall under the graph of the $g$ (resp.\ $h$) determines the value of $X$ (resp.\ $Y$).
   In the depicted situation, point 5 determines $X$ and point 4 determines $Y$ so that $X$ and $Y$ are unequal.
   In general, $X$ and $Y$ are equal whenever the first point to fall under the union of the graphs of $g$ and $h$ falls in their intersection. This occurs with
probability $1-F(\dTV(X,Y))$, and when $\mu$ has no atoms, it is the only way for $X$ and $Y$ to be equal.}
 \label{fig:coupling1}
\end{figure}

  Let $X$ and $Y$ be two of the variables with densities $g$ and $h$, respectively, and let $\alpha := \dTV(X,Y)$.
  To see that the disagreement probability is at most $F(\alpha)$, consider the first point $(x,s)$ to appear that has $s\leq g(x)\vee h(x)$.
  If it happens that $s\leq g(x)\wedge h(x)$, then we get $X=Y=x$.  
  Otherwise, this point determines the value of either $X$ or $Y$, and some later point determines the value of the other.
Consequently, for any measurable set $A$,
  \begin{equation}\label{eq:1}
  \P(X=Y \in A\text{ and the same point determines both $X$ and $Y$}) = \frac{\int_A (g \wedge h) d\mu}{\int_{\Omega} (g \vee h) d\mu} .
  \end{equation}
  Since $\int_{\Omega} (g\wedge h) d\mu = 1-\alpha$ and $\int_{\Omega} (g\vee h) d\mu = 1+\alpha$, we deduce that
  \[ \P(X=Y) \geq \frac{1-\alpha}{1+\alpha} = 1-F(\alpha) .\]

  Let $(x,s,t)$ be a point that determines one of $X$ or $Y$, but not the other.
  For continuous random variables, or more generally when $\mu$ has no atoms, the probability that the point $(x',s',t')$ that determines the other has $x=x'$ is zero, so that $\P(X\neq Y) = F(\alpha)$.
  When $\mu$ has atoms, this event may have a non-zero probability.
  The event that $X=Y=x$, and $X$ is determined by a point $(x,s,t)$ and $Y$ determined by a later point $(x,s',t')$ (i.e. $t'>t$) happens if and only if $h(x) < s \le g(x)$ and $s' \le h(x)$, and no earlier points determine $X$ or $Y$.
  The probability that the first point to determine $X$ or $Y$ determines $X$ but not $Y$ is $\frac{\alpha}{1+\alpha}$.
  Conditioned on this, $X$ and $Y$ are independent, with $X$ having density 
  $\frac {g-h}\alpha \mathbf{1}_{g>h} d\mu$, and with the law of $Y$ being unchanged.
  Thus,
  \[ \P(X=Y \in A\text{ and $X$ is determined before $Y$}) = \frac{1}{1+\alpha} \int_A (g-h) \mathbf{1}_{g>h} \cdot \P(Y=x) \, d\mu(x) .\]
  A similar formula holds when $g<h$ with $Y$ determined first.
  Combining the two, we get that the probability that $X=Y \in A$ but they are determined by distinct points is
  \[ \frac{1}{1+\alpha} \int_A |g-h| \cdot (\P(X=x)\wedge \P(Y=x)) \, d\mu(x). \]
  Rewriting the above integrand and using~\eqref{eq:1}, we obtain that
  \begin{equation}\label{eq:1b}
  \P(X=Y \in A) = \frac{1}{1+\alpha} \int_A (g \wedge h) \cdot (1+ |\P(X=x)-\P(Y=x)|) \, d\mu(x) ,
  \end{equation}
  from which the proposition follows.
\end{proof}

\subsection{Coupling \II}


We give now a second construction of a coupling of random variables.
We describe this construction for discrete random variables.
It is closely related to the so-called \emph{Poisson functional representation} which holds also for continuous random variables;
see the discussion after the proof.
We focus our discussion on the discrete case for several reasons: the discrete analysis is slightly simpler, \cref{T:main} was already proved in full generality using \cref{P:couplingI}, and \cref{T:main} can also be deduced from the discrete case by an approximation procedure.

\begin{prop}\label{P:couplingII}
 For any collection $\cS$ of random variables taking values in a common countable set, there exists a coupling such that, for any $X,Y \in \cS$,
\begin{equation}\label{eq:couplingII}
\P(X=Y) = \sum_u \left(\sum_v \frac{\P(X=v)}{\P(X=u)} \vee \frac{\P(Y=v)}{\P(Y=u)}\right)^{-1} .
\end{equation}
Moreover, this expression is at least $1-F(\dTV(X,Y))$.
\end{prop}

We emphasize that we do not assume that $\cS$ is a countable collection, but rather only that all random variables in $\cS$ are supported in a fixed countable set.
Indeed, our construction gives a coupling of \emph{all} random variables  supported in the given set.

\begin{proof}
  Suppose that the random variables take values in a countable set $U$.
  Let $\{E_u\}_{u\in U}$ be independent $\Exp(1)$ random variables.
  Fix a random variable $X\in\cS$ and denote $p_u := \P(X=u)$.
  Now define
  \[
    X := \argmin_{u \in U} \left\{\frac{E_u}{p_u} \right\},
  \]
  i.e., $X=u$ if $u$ is the minimizer of $\frac{E_u}{p_u}$.
  If there are multiple values of $u$ achieving the minimum, or if there is no minimizer, we pick a value for $X$ arbitrarily.
  Both of these are null events for any fixed $X \in \cS$. When the collection $\cS$ is uncountable, it may happen that there is always some variable in $\cS$ for which one of these events occurs, but this does not cause any problems.
  Standard properties of exponential variables imply that for any distribution $\{p_u\}$, the event that $\frac{E_u}{p_u}$ is smaller than $\frac{E_v}{p_v}$ for every $v\neq u$ has probability $p_u$.
  Thus, the variable $X$ constructed above has the required distribution, and therefore this defines a coupling of all the random variables in~$\cS$.

  We now show that this coupling satisfies~\eqref{eq:couplingII}.
  To this end, fix $X,Y\in\cS$ and denote $p_u := \P(X=u)$ and $q_u:=\P(Y=u)$ for $u \in U$.  
  Let us find an expression for $\P(X=Y=u)$ for a fixed $u \in U$.
  By the definition of the coupling, $\{X=Y=u\}$ is almost surely the event that $\frac{E_u}{p_u} \le \frac{E_v}{p_v}$ and $\frac{E_u}{q_u} \le \frac{E_v}{q_v}$ for every $v \in U$. Thus,
  \[ \P(X=Y=u) = \P\left( \frac{E_v}{E_u} \ge \frac{p_v}{p_u} \vee \frac{q_v}{q_u} \text{ for all }v \in U \right) .\]
  This is the probability that an exponential random variable with intensity 1 is the smallest among a family of independent exponential random variables with parameters $(\lambda_v)_v$, where $\lambda_v := \frac{p_v}{p_u} \vee \frac{q_v}{q_u}$ (note that $\lambda_u=1$). It then follows from standard properties of exponential random variables that this probability is $(\sum_v \lambda_v)^{-1}$. Hence,
  \begin{equation}\label{eq:2}
  \P(X=Y=u) = \left(\sum_v \frac{p_v}{p_u} \vee \frac{q_v}{q_u}\right)^{-1} .
  \end{equation}
  Summing over $u \in U$ yields~\eqref{eq:couplingII}.

  It remains to show that the right-hand side of~\eqref{eq:couplingII} is at least $1-F(\dTV(X,Y))$. To see this, we first observe that $1-F(x)=\frac{1-x}{1+x}$ and that
  \[
    \sum_u p_u \wedge q_u =  1-\dTV(X,Y)
    \qquad \text{and} \qquad
    \sum_u p_u \vee q_u =  1+\dTV(X,Y).
  \]
  Thus, it suffices to show that
  \begin{equation}\label{eq:3}
    \sum_u \left(\sum_v \frac{p_v}{p_u} \vee \frac{q_v}{q_u}\right)^{-1} \ge \frac{\sum_u p_u \wedge q_u}{\sum_{v} p_v \vee q_v} .
  \end{equation}
  This follows immediately from the inequality $\frac ab \vee \frac cd \le \frac{a \vee c}{b \wedge d}$.
\end{proof}

\paragraph{The Poisson functional representation.}
As noted above, this coupling is closely related to the \emph{Poisson functional representation} of Li and El Gamal \cite{TE,TA}, which was brought to our attention after online publication.
It is used there in the analysis of certain communication channels.
Let $X$ be a random variable with law $g d\mu$ for some $\sigma$-finite measure space $(\Omega,\mu)$.
Consider a Poisson point process $\cA$ with intensity $\mu\times\Leb$ on $\Omega\times\R_+$.
We denote the points of $\cA$ as $(x,s)$ with no third coordinate as in Coupling I.
We define $X=x_0$ if $(x_0,s_0)\in\cA$ minimizes $s/g(x)$ over the points of $\cA$.
Using the same Poisson process for a collection $\cS$ of random variables, all continuous w.r.t.\ $\mu$, yields a coupling of the variables.
See \cref{fig:coupling2}.

\begin{figure}
  \centering
  \includegraphics[scale=0.7,trim={2cm 0cm 1cm 0cm},clip]{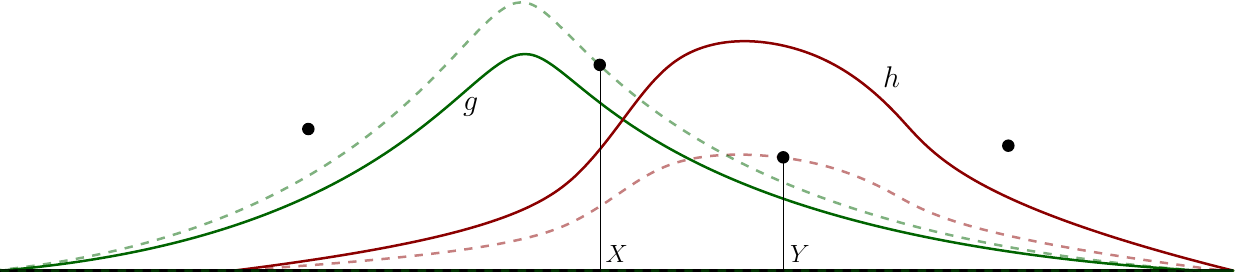}
  \caption{Illustration of Coupling \II.
    The densities $g$ and $h$ of $X$ and $Y$ are depicted.
    Points are a Poisson point process in the half-plane, but unlike in Coupling \rm{I}, there are no times associated with them.
    Also shown is the smallest multiple of $g$ (resp.\ $h$) that meet a point of the process.
    This intersection point determines the value of $X$ (resp.\ $Y$).
    In the depicted situation, different points determine $X$ and $Y$, and the variables are unequal.
  }
  \label{fig:coupling2}
\end{figure}


When $\mu$ is the counting measure on a countable set $U$, only the point $(x,s)$ with minima $s$ for each $x\in U$ is ever used in the coupling.
Since the $s$-coordinates of these points are i.i.d.\ $\Exp(1)$ random variables, we recover Coupling I in the discrete case.

One can deduce the disagreement bound for random variables constructed using this coupling (and hence \cref{T:main}) along the same lines as used in the proof of \cref{P:couplingII} above.
The main difference is that an additional step is required, to express $\P(X=Y)$ in terms of the Poisson process.
The analogue of \eqref{eq:couplingII} for variables $X$ and $Y$ with laws $gd\mu$ and $hd\mu$ is
\begin{equation}
  \label{eq:couplingIIc}
  \P(X=Y) = \int \left(\int \frac{g(v)}{g(u)} \vee \frac{h(v)}{h(u)} d\mu(v) \right)^{-1} d\mu(u).
\end{equation}
A lower bound for this probability can also be deduced from \cite[Lemma 1]{TA}.

\subsection{Comparison of the couplings}

The two coupling share various features beyond the fact that they both achieve the disagreement bound~$F$, but (except in degenerate cases) they are not the same coupling.

\paragraph{Geometric description of the couplings.}
While Coupling I is very intuitive and the fact that it achieves the disagreement bound $F$ is more transparent, there are good reasons to consider Coupling \II\ as well.
This is made clear by considering the case of random variables with common finite support.
Consider the collection $\cS$ of all random variables taking values in $\{1,\dots,n\}$.
The set $\cS$ is naturally described by the $(n-1)$-dimensional simplex $\Delta_n := \{ (a_1,\dots,a_n) \in [0,1]^n : \sum_i a_i =1 \}$ so that we may identify each random variable $X \in \cS$ with a point in $\Delta_n$.
A point in the simplex is a convex combination of the corners, and the coefficients (also referred to as barycentric coordinates) are the probabilities of the different values.
A coupling of the random variables in $\cS$ may be described as a random partition $A_1,\dots,A_n$ of the simplex so that $X = i$ for those variables $X\in A_i$. 
The validity of the coupling says that a point $X=(a_1,\dots,a_n)$ has $\P(X\in A_i) = a_i$ for all $i$.
The disagreement bounds are a control on the probability that nearby points (in the total-variation metric) are not in the same set of the partition.

Let us describe the two couplings using this terminology. 
Coupling I in this case is nothing but the Kleinberg--Tardos construction:
Each value $i\in\{1,\dots,n\}$ has a Poisson point process in $\R_+\times\R_+$, with points $(i,s,t)$. 
The value assigned to the random variable $X$ with coordinates $(a_1,\dots,a_n)$ is the $i$ associated with the point of minimal $t$ such that $s\leq a_i$.
We may clearly ignore points with $s>1$.
We can then think of the remaining points as arriving at random times, each with a random uniform $i$ and uniform $s\in[0,1]$. 
When at time $t$ we see a point $(i,s,t)$, the value $i$ is assigned to all $X$ with $a_i\geq s$ which have not already been assigned a value at an earlier time.
The set $\{X : a_i\geq s\}$ is a smaller simplex of size $s$ sharing the $i$-th corner of the full simplex.
An example is shown in \cref{fig:triangles}(a), where several such steps are visible.

\begin{figure}
  \centering
  \includegraphics[scale=0.6]{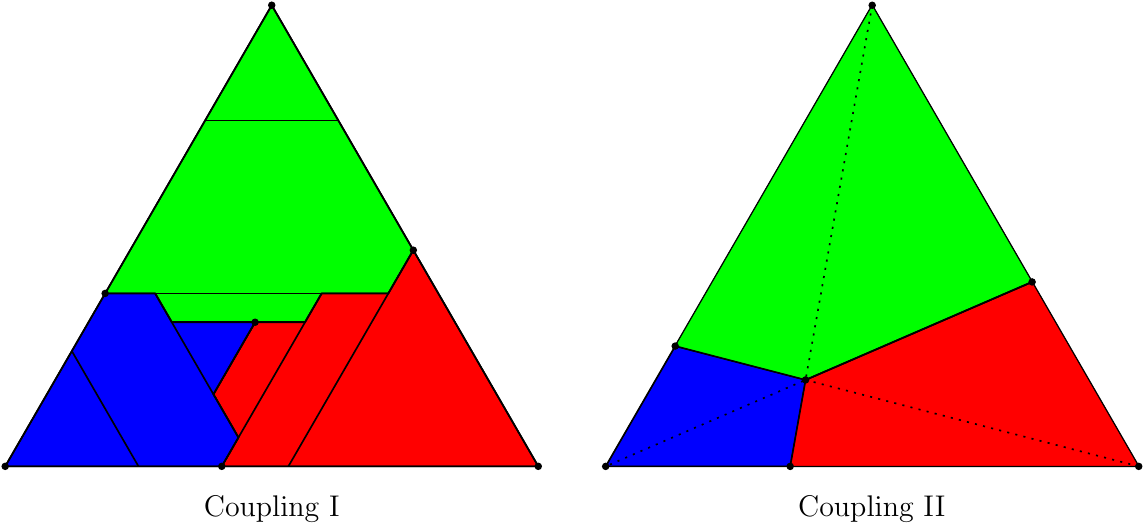}
  \caption{An illustration of the two couplings for three-valued random variables. 
    The associated partitions of the simplex are depicted. 
    (a) Coupling \rm{I} may be described by the overlapping ``territorial claims''. This is the Kleinberg--Tardos construction.
    (b) In Coupling~\II, the pivot point is uniform in the simplex. This is a special case of the Poisson functional representation of Li and El Gamal.}
  \label{fig:triangles}
\end{figure}

While Coupling I is very simple to describe and understand, the resulting partition of the simplex is evidently somewhat complex. 
In particular, the parts of the partition are not necessarily convex (though they are star-like).
Coupling~\II, while less transparent in its construction, yields a remarkably simple partition.
The interfaces between the parts $A_i$ are given by relations on the ratios, with $A_i$ adjacent to $A_j$ where $E_i/a_i = E_j/a_j$ (where $E_i$ are the exponentials used in the construction). 
This is a hyperplane passing through all but two vertices of the simplex.
Indeed, the entire partition is determined by a unique point $U$ where $E_i/a_i$ is the same for all $i$. 
Since $E_i$ are independent exponential random variables, $U$ is a uniform point in the simplex. 
The hyperplanes passing though $U$ and any $n-2$ of the corners give the partition of the simplex.
This is shown in \cref{fig:triangles}(b).

\paragraph{Sharpness of couplings.}
Both Coupling I and Coupling \II\ satisfy that $\P(X \neq Y) \le F(\dTV(X,Y))$ for any two random variables $X$ and $Y$, and for both constructions there are pairs of random variables for which they do no better.

For Coupling I, in the case of continuous random variables with respect to the Lebesgue measure, or for any $\mu$ with no atoms, Coupling I achieves the disagreement bound $F$ precisely, and no better.
As remarked above, we can also use Coupling I in the discrete case, where the random variable $X_i$ has density $f_i$ with respect to the counting measure.
In this case, it is possible that $X=Y$ even if distinct points $(x,s,t)$ and $(x',s',t')$ determine their value, since it may happen that $x=x'$. Indeed, the second term on the right-hand side of~\eqref{eq:coup1_pr} is zero if and only if, for every $x$, either $\P(X=x)$ and $\P(Y=x)$ are equal or one of them is zero.

For Coupling \II, suppose that $\cS$ consists of discrete random variables taking values in $U$.
An inspection of the inequality used in~\eqref{eq:3} reveals that there is equality in~\eqref{eq:3} if and only if $p_u=q_u$ or $p_u \wedge q_u = 0$, which is the same condition as for Coupling I.
In any other case, both couplings yield a disagreement probability which is strictly smaller than $F(\dTV(X,Y))$. 

\paragraph{Comparison of the disagreement probabilities.}
Since the two couplings achieve the worst-case disagreement probability $F$ in the same cases, it is natural to ask how they compare in general.
It turns out that Coupling \II\ is not only geometrically simpler as seen in \cref{fig:triangles}, but also achieves better disagreement probabilities than Coupling I for any pair of discrete random variables.
In fact, for any two discrete random variables $X$ and $Y$ and any value $u$, the probability that $X=Y=u$ is at least as large under Coupling~\II\ than under Coupling~I. This is seen by comparing the formulas~\eqref{eq:1b} and~\eqref{eq:2}. Denote $p_v := \P(X=v)$ and $q_v := \P(Y=v)$. We must show that
\[ (1+|p_u-q_u|) \cdot (p_u \wedge q_u) \cdot \sum_v \frac{p_v}{p_u} \vee \frac{q_v}{q_u} \le \sum_v p_v \vee q_v .\]
Suppose without loss of generality that $p_u \le q_u$ and consider the set $S := \{ v : \frac{p_v}{p_u} \le \frac{q_v}{q_u} \}$. Then $p_v \le q_v$ for $v \in S$, so that $\sum_v p_v \vee q_v \ge 1 + \sum_{v \in S} (q_v - p_v)$. It thus suffices to show that
\[ (1 + q_u - p_u) \cdot \left(\frac{p_u}{q_u} \sum_{v \in S} q_v + \sum_{v \notin S} p_v\right) \le 1 + \sum_{v \in S} (q_v - p_v) .\]
Using that $\sum_{v \notin S} p_v = 1 - \sum_{v \in S} p_v$, we see that it suffices that
\[ q_u - p_u \le \sum_{v \in S} \left(q_v - p_v - (1 + q_u - p_u)(\tfrac{p_u q_v}{q_u} - p_v) \right) .\]
Using the assumption that $p_u \le q_u$ and rewriting the summand as $\frac{1}{q_u}(q_u-p_u)(q_v-p_uq_v+p_vq_u)$, we see that every term in the sum is non-negative.
Since $u \in S$, the inequality is easily seen to hold.
(In fact, the inequality is strict for some $u$'s, except for very simple cases.)

In the continuous setting, it is even easier to see that Coupling \II achieves a smaller disagreement probability than Coupling I.
Indeed, Coupling I gives a disagreement probability exactly equal to $F(\dTV(X,Y))$, while the inequality in the continuous version of \eqref{eq:3} is in general a strict inequality.

\paragraph{$k$-tuple disagreements.}
We have shown that both couplings are ``nearly optimal'' for disagreements among pairs of random variables. In fact, both couplings are also nearly optimal (in a similar sense) for disagreements among $k$-tuples of random variables. 
Namely, for any $k$ random variables $X_1,\dots,X_k$, the probability they are not all equal under either coupling is comparable to its smallest possible value $\alpha := 1-\sum_u \P(X_1=u) \wedge \dots \wedge \P(X_k=u)$ (given by the optimal coupling of $X_1,\dots,X_k$ and no others).
Precisely, under either coupling, we have
\[ \P(X_1,\dots,X_k\text{ are not all equal}) \le \frac{k\alpha}{1+(k-1)\alpha} \le k \alpha .\]
This follows from
\[ \P(X_1=\dots=X_k) \ge \frac{\sum_u \P(X_1=u) \wedge \dots \wedge \P(X_k=u)}{\sum_u \P(X_1=u) \vee \dots \vee \P(X_k=u)} ,\]
which can be shown for Coupling~I by a similar computation as in~\eqref{eq:1} and for Coupling~\II\ by a similar computation as in~\eqref{eq:2} and~\eqref{eq:3}.

For certain collections $\cS$ of random variables, the latter bound cannot be improved.
For example, consider the set $\cS$ of $n \ge k$ random variables $S_1,\dots,S_n$ such that each $S_i$ is uniform on $\{1,\dots,n\} \setminus \{i\}$.
In any coupling of $S_1,\dots,S_n$, there exists a subset $X_1,\dots,X_k$ of the random variables for which the reverse inequality holds.
To see this, note that the number of subsets $I \subset \{1,\dots,n\}$ of size $k$ for which not all $\{S_i\}_{i \in I}$ are equal is always at least $\binom{n-1}{k-1}$.
Thus, in any coupling, there must be such a subset $I$ for which the probability of this event is at least $\binom{n-1}{k-1} / \binom{n}{k} = k/n$.
On the other hand, for any $k$ of the random variables $X_1,\dots,X_k$,
\[ \frac{\sum_u \P(X_1=u) \wedge \dots \wedge \P(X_k=u)}{\sum_u \P(X_1=u) \vee \dots \vee \P(X_k=u)} = \frac{(n-k)/(n-1)}{n/(n-1)} = 1 - \frac{k}{n} .\]

\section{Optimality of disagreement bounds}
\label{sec:optimality}

In this section, we investigate the optimality of \cref{T:main}.
As noted, it is natural to ask whether there are any disagreement bounds smaller than $F$.
The first set of results are lower bounds on $f(x)$ for any single $x$, and we do not believe these are optimal for generic $x$.
The second set of results lead to \cref{T:optimal-bound}, which states that there is no disagreement bound that is less than $F$ globally.

\subsection{Local optimality of $F$}

The trivial lower bound (see \cref{T:couple2}) is that any disagreement bound must have $f(x)\geq x$ for all $x$.
The example presented just before \cref{T:main}, of three variables each taking two possible values, shows that any disagreement bound must have $f(\frac 12)\geq \frac 23 = F(\frac 12)$.
This is generalized by the following.

\begin{prop}\label{prop:1/n-bound}
  Any disagreement bound $f$ must have
  \[
    f(\tfrac1n) \geq F(\tfrac1n) = \tfrac{2}{n+1} \qquad\text{for any integer $n \ge 1$}.
  \]
  In particular, $ax$ is a disagreement bound for $a=2$, but not for any smaller $a$.
\end{prop}

\begin{proof}
  Consider the case when $\cS$ consists of $n+1$ random variables $X_0,\dots,X_n$, where each $X_i$ is uniform on $\{0,\dots,n\} \setminus \{i\}$.
  Then $\dTV(X_i,X_j)=\frac{1}{n}$ for any $i \neq j$.
  However, it is impossible for all variables $X_i$ to be equal, and therefore at least $n$ of the $\binom{n+1}{2}$ pairs must disagree.
  Thus, under any coupling,
  \[
    \sum_{i < j} \P(X_i \neq X_j) \ge n,
  \]
  and hence, $\P(X_i\neq X_j) \geq \frac{2}{n+1}$ for some $i \neq j$.
\end{proof}

The above proposition shows that $F$ provides the best possible value for a disagreement bound at any inverse integer. We do not know whether an analogous statement holds at every point $x \in (0,1)$. Nevertheless, we are able to provide a lower bound at any point $x$, which improves on the trivial lower bound $x$, and nearly matches $F(x)$ for small $x$.
See \cref{fig:bounds} for a comparison between $F$ and our lower bounds.

\begin{prop}\label{prop:all-x-bound}
  Any disagreement bound $f$ must have
  \[
  f(x) \ge \frac{2x}{1+\frac{1}{\lfloor \frac1x \rfloor}}
  \qquad\text{for any $x \in (0,1)$}.
  \]
  In particular, any disagreement bound $f$ satisfies $\liminf_{x \to 0} \frac{f(x)}{x} \ge 2$.
\end{prop}

\begin{proof}
  We use a variant of the construction from the proof of \cref{prop:1/n-bound}.
  Fix $x \in (0,1)$, $n \ge 1$ and $\eps \ge 0$ such that $n(x+\eps)=1-\eps$.
  Consider the case when~$\cS$ consists of $n+1$ random variables $X_0,\dots,X_n$, where each $X_i$ takes the value $i$ with probability $\eps$, and takes any other value with probability $x+\eps$.
  Then $\dTV(X_i,X_j)=x$ for any $i \neq j$.
  Let $\P$ be some coupling of these variables. Observe that the variables $X_i$ are all equal with probability at most $(n+1)\eps = 1-nx$. 
  Thus, with probability at least $nx$, they are not all equal, in which case at least $n$ of the $\binom{n+1}{2}$ pairs must disagree. Therefore,
  \[
    \sum_{i<j} \P(X_i \neq X_j) \ge n^2 x,
  \]
  and hence, $\P(X_i\neq X_j) \geq \frac{2nx}{n+1}$ for some $i \neq j$.
  Taking the largest $n$ compatible with a given $x$, namely $n=\lfloor \frac1x \rfloor$, yields the inequality.
\end{proof}

\begin{figure}
  \centering
  \includegraphics[width=.5\textwidth]{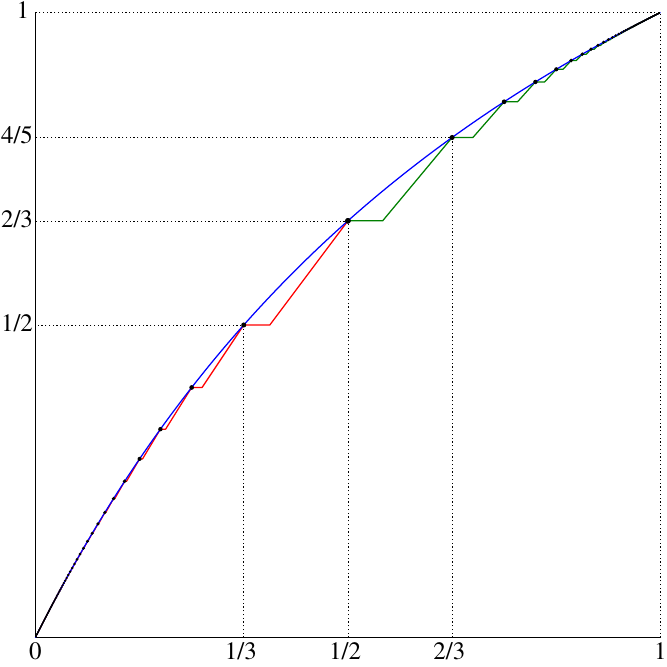}
  \caption{The disagreement bound $F(x)$ is shown in blue.
    Our pointwise lower bound on any disagreement bound is in red and green.
    The bullet points are from \cref{prop:1/n-bound,prop:near1}. The red segments are from~\eqref{eq:pointwise-bound1} and the green from~\eqref{eq:pointwise-bound2}. Both are obtained by interpolating the bullet points using \cref{cor:pointwise-bound}.
    Only the better of the two bounds is plotted.
  }
  \label{fig:bounds}
\end{figure}

One may think of Proposition~\ref{prop:all-x-bound} as converting the pointwise bound at $x=\frac 1n$ from Proposition~\ref{prop:1/n-bound} to a slightly worse pointwise bound at any point $x<\frac 1n$.
In fact, a similar perturbation argument shows that any pointwise lower bound can be converted to a slightly worse pointwise lower bound at any other point. This will be a simple consequence of the following.

\begin{prop}\label{prop:perturbed-disagreement-bound}
  Let $f$ be a disagreement bound and let $0 \le \delta \le \eps<1$.
  Define
  \[ \tilde{f}(x) := \frac{f((1-\eps)x + \delta)}{1-\eps+\delta} .\]
  Then $\tilde{f} \wedge 1$ is also a disagreement bound.
\end{prop}

\begin{proof}
 Let $\cS$ be a finite collection of random variables.
Let $U$ consist of those $u$ such that $\P(X=u)=0$ for all $X \in \cS$, and choose elements $\{u(X)\}_{X \in \cS}$ in $U$ and an additional element $u \in U$, all distinct from each other.
In order to use that $f$ is a disagreement bound, for each $X \in \cS$, we define a new random variable $X'$ by letting $X'$ equal $X$ with probability $1-\eps$, equal $u(X)$ with probability $\delta$, and otherwise equal $u$. Note that $\dTV(X',Y') = (1-\eps) \dTV(X,Y) + \delta$. Since $f$ is a disagreement bound, there exists a coupling of the prime variables so that $\P(X' \neq Y') \le f(\dTV(X',Y'))$ for any $X'$ and $Y'$.

To show that $\tilde f \wedge 1$ is a disagreement bound, we need to exhibit a coupling of the (original) variables in $\cS$ so that $\P(X \neq Y) \le \tilde{f}(\dTV(X,Y))$ for any $X,Y \in \cS$. Towards constructing such a coupling, consider a sequence of independent samples from the above coupling of the prime variables, and let $\{X^n\}_{X \in S, n \ge 1}$ denote these samples. Now take $X$ to equal $X^n$, where $n$ is the smallest index such that $X^n \notin \{u,u(X)\}$. It is straightforward that this indeed yields a coupling of the variables in $\cS$.
To see that it satisfies the required bound on the disagreement probabilities, fix two variables $X,Y \in \cS$ and note that $X \neq Y$ implies that either $X^1 \neq Y^1$ or $X^1=Y^1=u$. Thus,
\[ \P(X \neq Y) \le \P(X^1 \neq Y^1) + \P(X \neq Y \mid X^1=Y^1=u) \cdot \P(X^1=Y^1=u) .\]
Since $\P(X \neq Y \mid X^1=Y^1=u) = \P(X \neq Y)$, we obtain that
\[ \P(X \neq Y)
  \le \frac{\P(X^1 \neq Y^1)}{1-\P(X^1=Y^1=u)}
  \le \frac{f(\dTV(X',Y'))}{1-\eps+\delta}
  = \tilde{f}(\dTV(X,Y)) . \qedhere \]
\end{proof}

\begin{corollary}\label{cor:pointwise-bound}
  Let $g$ be the pointwise infimum over all disagreement bounds $f$.
  Then $g(x)$ is non-decreasing and $g(x)/x$ is non-increasing.
  Moreover, $g$ is Lipschitz continuous.
\end{corollary}

\begin{proof}
  By Proposition~\ref{prop:perturbed-disagreement-bound}, we have
  $g(x) \le \frac{g((1-\eps)x+\delta)}{1-\eps+\delta}$ for any $x \in (0,1)$ and $0 \le \delta \le \eps < 1$.
  Taking $\delta=\eps$ yields that $g(x) \le g(x+\eps(1-x))$ for any $\eps$,
  which shows that $g(x)$ is non-decreasing.
  Taking $\delta=0$ and $\eps=s/x$ yields that $g(x) \le g(x-s)/(1-s/x)$ for $s \in (0,x)$, which shows that $g(x)/x \leq g(x-s)/(x-s)$, and hence $g(x)/x$ is non-increasing.
  
  Using the two monotonicity properties and the fact that $g \le F$ by \cref{T:main}, we obtain that
  \[ 0 \le \frac{g(y)-g(x)}{y-x} \le \frac{g(x)}x \le \frac{2}{1+x} \leq 2 \qquad\text{for any }0\le x < y \le 1 .\]
  In particular, $g$ is Lipschitz continuous.
\end{proof}

Let us summarize the bounds we have shown in this section.
Let $g$ be as in the corollary above. \cref{prop:1/n-bound} tells us that $g(x)=F(x)$ for $x=\frac1n$ and any integer $n \ge 1$. \cref{cor:pointwise-bound} allows us to interpolate these to get a lower bound on $g$ at any point. Specifically, it shows that for any integer $n \ge 2$, we have
  \begin{equation}\label{eq:pointwise-bound1}
  g(x) \geq \max\left( \frac{2x(n-1)}{n}\ ,\  \frac{2}{n+1} \right) \qquad\text{for any }x\in[\tfrac{1}{n},\tfrac{1}{n-1}] .
  \end{equation}
Similarly, \cref{prop:near1} below shows that $g(x)=F(x)$ also holds for $x=1-\frac1n$ and any integer $n \ge 1$. Consequently, \cref{cor:pointwise-bound} implies that for any integer $n \ge 2$, we have
  \begin{equation}\label{eq:pointwise-bound2}
     g(x) \geq \max\left( \frac{2xn}{2n-1}\ ,\  \frac{2n-4}{2n-3} \right) \qquad\text{for any } x\in[1-\tfrac{1}{n-1},1-\tfrac{1}{n}] .
  \end{equation}
These bounds are depicted in \cref{fig:bounds}.

\subsection{Combinatorial improvements}

By more careful combinatorial analysis, we get the following extensions of \cref{prop:1/n-bound} which show that $F(x)$ is a lower bound at rational points in which the denominator is large in comparison to the numerator, as well as at $1-1/n$.
\cref{prop:k/n-bound} is proved here, while the proof of \cref{prop:near1} is deferred to the end of \cref{sec:comb2}.

\begin{prop}\label{prop:k/n-bound}
  Any disagreement bound $f$ must have
  \[
  f(\tfrac kn) \geq F(\tfrac kn) = \tfrac{2k}{n+k} \qquad\text{for any integers $k \ge 2$ and $n \ge 3k^2+6k$}.
  \]
\end{prop}

\begin{prop}\label{prop:near1}
  Any disagreement bound $f$ must have
  \[
  f(1-\tfrac1n) \geq F(1-\tfrac1n) = \tfrac{2n-2}{2n-1} \qquad\text{for any integer }n \ge 2. 
  \]
\end{prop}

We next introduce a combinatorial lemma which we require for the proof of \cref{prop:k/n-bound}.
Fix integers $k,n \ge 1$.
Let $S_{n,k}$ be the collection of all $I\subset\{1,\dots,n+k\}$ of size $n$.
The Hamming \textbf{distance} between two sets $I$ and $J$ in $S_{n,k}$ is defined by
\[ d(I,J) := |I\setminus J| = |J\setminus I| .\]
A \textbf{$(n,k)$-assignment} is a selection of an element from each $I$ in $S_{n,k}$, namely, $z=(z_I)_{I\in S_{n,k}}$ with $z_I \in I$.
For such an assignment, we denote by $D_m(z)$ the number of \textbf{distance-$m$ disagreements}, defined by
\[ D_m(z) := \# \Big\{ (I,J) : d(I,J)=m \text{ and }z_I \neq z_J \Big\} .\]

In the remainder of this section, we focus on the case when $k \le n$ and $m=k$. A pair of sets $I$ and $J$ in $S_{n,k}$ are called \textbf{distant} if $d(I,J) = k$ (this is the maximal possible distance when $k\le n$).
Note that these are ordered pairs, and that the total number of distant pairs is the multinomial coefficient
\[ \binom{n+k}{k,k} = \frac{(n+k)!}{k!^2(n-k)!} .\]
Using this notation, the proof of \cref{prop:1/n-bound} relied on the simple fact that any $(n,1)$-assignment has at least $n$ distant disagreement pairs.
Equivalently, at least a $F(\frac1n)$-fraction of distant pairs disagree.
The following lemma shows that, when $n$ is large is comparison to $k$, any $(n,k)$-assignment has at least a $F(\frac kn)$-fraction of distant disagreements.

\begin{lemma}\label{lem:disagreement-large-n}
  Let $a=1/\log \frac 32 \approx 2.466$.
  For any $k \ge 2$, $n \ge ak^2+6k$ and $(n,k)$-assignment $z$, we have
  \begin{equation}\label{eq:disjoint-disagreement-bound}
    D_k(z) \ge \frac{2k}{n+k} \cdot \binom{n+k}{k,k} .
  \end{equation}
\end{lemma}

The condition $n\geq ak^2+O(k)$ is an artifact of the following proof, and can no doubt be improved.
The bound \eqref{eq:disjoint-disagreement-bound} is motivated by the idea that up to a permutation of the elements $\{1,\dots,n+k\}$, the way to minimize disagreements is to take $z_I = \min(I)$, for which there is equality in~\eqref{eq:disjoint-disagreement-bound}.
We do not know what the minimal $n$ above which this assignment minimizes $D_k(z)$ is. 
We note however that~\eqref{eq:disjoint-disagreement-bound} does not necessarily hold for small~$n$. 
For example, the $(3,2)$-assignment $z$ given by $z_{\{i,j,k\}} = 2(i+j+k) \pmod{5}$ has $|D_2(z)|=20$, compared to $24$ for $z_I=\min(I)$.

\begin{proof}
  At the heart of the proof is the observation that, when $n$ is large enough, most sets $I$ contain any given element and most pairs of sets are distant.
  Suppose for a contradiction that there is a counterexample to the lemma, and let $z$ be an assignment with minimal possible $D_k(z)$.
  Without loss of generality, we assume that the most common value among the $z_I$ is $1$, and let
  \[ N_1 := \# \{I : z_I=1\} \]
  be the number of times it appears.
  Since each $z_I$ agrees with at most $N_1$ other variables, by considering distant pairs $(I,J)$ having $z_I=z_J$, it is clear that
  \[ \binom{n+k}{k,k} - D_k(z) \le N_1 \cdot \binom{n+k}{k} .\]
  Therefore,
  \begin{equation}\label{eq:N-lower-bound}
    N_1 
    \ge \frac{\left(1-\frac{2k}{n+k}\right)\binom{n+k}{k,k}}{\binom{n+k}{k}} 
    = \frac{n-k}{n+k} \cdot \binom{n}{k} .
  \end{equation}
  Note that, if $n$ is large enough, this shows that most $I$ have $z_I=1$.

  We claim that minimality of $D_k(z)$ implies that every $I$ such that $1\in I$ has $z_I=1$. 
  Indeed, suppose some $I$ has $1\in I$ and $z_I\neq 1$, and consider an assignment $z'$ which equals $z$ except that $z'_I=1$.
  This modification introduces at most $2\left[\binom{n+k}{k}-N_1\right]$ new distant disagreement pairs (the $2$ is since these are ordered pairs).
  Among the total $\binom{n+k}{k}$ sets $J$, there are $\binom{n-1}{k}$ sets $J$ that are both distant from $I$ and contain $1$. 
  Among those, at most $\binom{n+k}{k}-N_1$ do not have $z_J=1$. 
  Thus, the number of eliminated distant disagreement pairs is at least
  $2\left[\binom{n-1}{k} + N_1 - \binom{n+k}{k}\right]$.
  This contradicts minimality of $D_k(z)$ when
  \[ \binom{n+k}{k}-N_1 < \binom{n-1}{k} + N_1 - \binom{n+k}{k} .\]
  In light of \eqref{eq:N-lower-bound} this holds when
  \[ \binom{n+k}{k} < \frac12 \binom{n-1}{k} + \frac{n-k}{n+k}\binom{n}{k} .\]
  Using $\binom{n+k}{k} \leq e^{k^2/(n-k)} \binom{n}{k}$, this is seen to hold when $e^{k^2/(n-k)} < \frac{(n-k)(3n+k)}{2n(n+k)}$, which in turn holds for every $n\geq ak^2 + 6k$.
  
  Thus, we have proved that a counterexample with minimal $D_k(z)$ has $z_I=1$ for every $I$ with $1\in I$.
  Hence, the number of distant disagreement pairs is at least twice the number of distant pairs $(I,J)$ having $1\in I$ and $1\not\in J$, the latter being
  $\binom{n+k-1}{k,k-1} = \frac{k}{n+k} \binom{n+k}{k,k}$.
\end{proof}

We are now ready to prove \cref{prop:k/n-bound}.

\begin{proof}[Proof of \cref{prop:k/n-bound}]
  The proof uses yet another variant of the construction from the proof of \cref{prop:1/n-bound}.
  Let $\cS$ consist of $\binom{n+k}{k}$ random variables $\{X_I\}_{I \subset \{1,\dots,n+k\}, |I|=n}$, where each $X_I$ is uniform on $I$.
  Then $\dTV(X_I,X_J)=\frac{k}{n}$ for distant $I$ and $J$.
  Let $\P$ be any coupling of these variables.
  Since $X=(X_I)_I$ is a $(n,k)$-assignment, \cref{lem:disagreement-large-n} implies that, almost surely,
  \[ D_k(X) \ge F\left(\tfrac kn\right) \cdot \binom{n+k}{k,k} .\]
  In other words, the fraction of distant pairs $(I,J)$ with $X_I\neq X_J$ is at least $F(\frac kn)$.
  Thus, there exist distant $I$ and $J$ such that $\P(X_I\neq X_J) \geq F(\frac kn)$.
\end{proof}

\subsection{Global optimality of $F$}
\label{sec:comb2}

Our goal now is to prove \cref{T:optimal-bound}.
The main step is the following.

\begin{prop}\label{prop:optimal-bound-step}
  Let $n,k \ge 1$, and let $c_m$ be the probability that $d(I,J)=m$, where $I$ and $J$ are two independently chosen uniform subsets of $\{1,\dots,n+k\}$ of size $|I|=|J|=n$.
  Then any disagreement bound $f$ satisfies
  \[ \sum_{m=1}^{k \wedge n} c_m f(\tfrac {\alpha m}n) \ge \alpha \sum_{m=1}^{k \wedge n} c_m F(\tfrac mn) \qquad\text{for all } 0 \le \alpha \le 1.\]
\end{prop}

Let us see how \cref{T:optimal-bound} follows from \cref{prop:optimal-bound-step}.

\begin{proof}[Proof of \cref{T:optimal-bound}]
  Let $f$ be a disagreement bound such that $f \le F$ on an interval $(a,b) \subset [0,1]$.
  We must show that $f$ coincides with $F$ on $(a,b)$. To illustrate the basic idea behind the proof, observe that if $a=0$, then \cref{prop:optimal-bound-step} easily yields that $f(x)=F(x)$ for every rational $x \in (0,b)$. Indeed, if $k/n<b$ is a rational number such that $f(\frac{k}{n}) < F(\frac{k}{n})$, then the proposition is violated with that $k$ and $n$ (and with $\alpha=1$).
  Since there is no obvious monotonicity or continuity for disagreement bounds, the proof below relies on a perturbative argument to address the case of irrational points and of $a>0$.
  
  Let $x \in (a,b)$. 
  For an arbitrary $n \ge 1$, let $m$ be such that
  \[ \tfrac{m-1}n < x \le \tfrac mn .\]
  Set $\alpha := \frac{nx}m$ so that $x=\frac {\alpha m}n$ and $0\le 1-\alpha<\frac1m < \frac1{nx}$.
  We henceforth regard $x,m,n,\alpha$ as fixed, and we aim to choose a suitable $k$ for which to apply \cref{prop:optimal-bound-step}.
  By standard estimates (using Sterling's approximation), there exists some $k \ge m$ such that
  \[ c_m \ge \frac c{\sqrt n} \qquad\text{and}\qquad \sum_{\frac{am}x<i<\frac{bm}x} c_i \ge 1-Ce^{-cn},\]
  where $C,c>0$ are constants which do not depend on~$n$. 
  With this choice of $k$, by \cref{prop:optimal-bound-step},
  \[ \sum_{i=1}^{k \wedge n} c_i f\left(\tfrac {ix}m \right) \ge \big(1-\tfrac{1}{nx} \big) \cdot \sum_{i=1}^{k \wedge n} c_i F\left(\tfrac in \right) .\]
  Since $f(\frac{ix}m) \le F(\frac{ix}m)$ for all $\frac{am}x < i < \frac{bm}x$ and since $f,F \le 1$, we obtain that
  \begin{align*}
  c_m(F(x) - f(x)) 
   &\le \sum_{i=1}^{k \wedge n} c_i \left( F\left(\tfrac {ix}m \right) - f\left(\tfrac {ix}m \right) \right) + Ce^{-cn} \\
   &\le \tfrac{1}{nx} \cdot \sum_{i=1}^{k \wedge n} c_i F\left(\tfrac in \right) + Ce^{-cn} \le \tfrac{1}{nx}+Ce^{-cn} .
  \end{align*}
  By the lower bound on $c_m$, we have $F(x)-f(x) \le \frac{1}{c\sqrt{n}x} + \frac Cc \sqrt{n}e^{-cn}$. 
  Since $n$ may be taken arbitrarily large, it follows that $f(x)=F(x)$.
\end{proof}

The proof of \cref{prop:optimal-bound-step} requires additional combinatorial lemmas.
We have seen in \cref{lem:disagreement-large-n} that, when $n$ is large is comparison to $k$, any $(n,k)$-assignment has at least a $F(\frac kn)$-fraction of distant disagreements, i.e., $D_k(z) \geq F(\frac kn) \binom{n+k}{k,k}$.
The analogous statement for distance-$m$ pairs is that the number of distance-$m$ disagreements is at least a $F(\frac mn)$-fraction of all distance-$m$ pairs, i.e.,
\[ D_m(z) \ge F\left(\tfrac mn\right) \cdot \binom{n+k}{m,m,k-m} .\]
While we have no proof of this inequality for any particular $m$, the following lemmas establish a linear combination of these bounds for different $m$'s.

\begin{lemma}\label{lem:disagreement}
  For any $n,k \ge 1$ and any $(n,k)$-assignment $z$, we have
  \begin{equation}\label{eq:agreement-bound2}
    \sum_{m=1}^{k \wedge n} D_m(z) \ge \sum_{i=0}^{k-1} 2\binom{n+i}{i+1} \binom{n+i}{i} .
  \end{equation}
\end{lemma}

\begin{proof}
  We seek a lower bound on the total number of disagreements $D:=\sum_{m=1}^{k \wedge n} D_m(z)$.
  For $i \in \{1,\dots,n+k\}$, let $N_i$ be the number of variables $z_I$ that equal~$i$.
  Without loss of generality, we may assume that $N_1 \ge N_2 \ge \dots \ge N_{n+k}$.
  The number of disagreements is precisely
  \[ D = \sum_{i \neq j} N_i N_j = \binom{n+k}{k}^2 - \sum_i N_i^2 .\]
  We claim that the above is minimized by the ``greedy'' assignment $z_I = \min(I)$ which has
  \[ N_i = \binom{n+k-i}{k-i+1} \qquad \text{for $i\leq k+1$} ,\]
  and $N_i=0$ for $i>k+1$. 
  Indeed, minimizing $D$ is equivalent to maximizing $\sum_i N_i^2$. 
  To see that the greedy choice maximizes this latter quantity, note that if 
  $a\geq b$ then $(a+1)^2 + (b-1)^2 > a^2+b^2$.  
  Since $N_i$ are decreasing, if $z_I \neq \min(I)$ for some $I$, then decreasing $z_I$ will increase $\sum_i N_i^2$.
  
  Finally, the number of $I$ with $\min(I)>i$ is $\binom{n+k-i}{k-i}$, and so for the greedy assignment we have
  \[ D = \sum_i 2N_i \sum_{j>i} N_j 
  = \sum_{i=1}^k 2\binom{n+k-i}{k-i+1}\binom{n+k-i}{k-i} . \]
  The lemma follows after a change of the index of summation.
\end{proof}

\begin{lemma}\label{lem:combi-identity} 
  For any $n,k \ge 1$, we have
  \[ \sum_{i=0}^{k-1} \binom{n+i}{i+1} \binom{n+i}{i} = \sum_{m=1}^{k\wedge n} \frac{m}{n+m} \cdot \binom{n+k}{m,m,k-m} .\]
\end{lemma}

\begin{proof}
Let $E$ denote the set of ordered pairs $(I,J)$ of subsets of $\{1,\dots,n+k\}$ such that $|I|=k$, $|J|=k-1$ and $\min (I^c) = \min (J^c)$. Here and below, all complements are taken within $\{1,\dots,n+k\}$. We show that both sides of the desired equality count the number of elements in $E$.

We begin with the left-hand side.
Since the $i$-th term in the sum is easily seen to count the number of $(I,J) \in E$ such that $\min (I^c) = k-i$, it follows that the left-hand side equals $|E|$.

We now turn to the right-hand side, which may be rewritten as
\[ \sum_{m=1}^{k\wedge n} \binom{n+k}{k-m} \binom{n+m-1}{m-1} \binom{n}{m} .\]
Let us show that the $m$-th term in the sum counts the number of $(I,J) \in E$ such that $|I \setminus J|=m$. Indeed, $\binom{n+k}{k-m}$ is the number of ways to choose $S=I \cap J$, and given any such choice, noting that neither $I$ nor $J$ can contain the number $s := 1+\min (S^c)$, we see that $\binom{n+m-1}{m-1}$ is the number of ways to choose $J \setminus S$ (which must be disjoint from $S \cup \{s\}$) and $\binom nm$ is then the number of ways to choose $I \setminus S$ (which must be disjoint from $J \cup \{s\}$).
\end{proof}

We are now ready to prove \cref{prop:optimal-bound-step}.

\begin{proof}[Proof of \cref{prop:optimal-bound-step}]
  We first address the case when $\alpha=1$. For this we use the same construction as in the proof of \cref{prop:k/n-bound}.
  Let $\cS$ consist of $\binom{n+k}{k}$ random variables $\{X_I\}_{I \subset \{1,\dots,n+k\}, |I|=n}$, where each $X_I$ is uniform on $I$.
  Then
  \[ \dTV(X_I,X_J) = \tfrac{d(I,J)}{n} \qquad \text{for any $I$ and $J$}.\]
  By \cref{lem:disagreement,lem:combi-identity}, under any coupling of the variables in $\cS$, almost surely,
  \[ \sum_{m=1}^{k \wedge n} D_m(X) 
  \ge \sum_{m=1}^{k \wedge n} F\left(\tfrac mn\right) \cdot \binom{n+k}{m,m,k-m}
  = \sum_{I,J} F\left(\dTV(X_I,X_J)\right) .\]
  Hence, by considering a coupling $\P$ for which $\P(X_I \neq X_J) \le f(\dTV(X_I,X_J))$ for all $I$ and $J$, and taking expectation, we obtain that
  \[ \sum_{I,J} f\left(\dTV(X_I,X_J)\right) \ge \sum_{I,J} \P(X_I \neq X_J) \ge \sum_{I,J} F\left(\dTV(X_I,X_J)\right) .\]
  Since the fraction of the terms where $\dTV(X_I,X_J)=\frac mn$ is $c_m$, this establishes the proposition in the case $\alpha=1$ (note that we may assume that $f(0)=0$).

  The case when $0 \le \alpha<1$ now follows by applying the case $\alpha=1$ to the disagreement bound $\frac{f(\alpha x)}{\alpha} \wedge 1$ given by \cref{prop:perturbed-disagreement-bound} with $\eps=1-\alpha$ and $\delta=0$.
\end{proof}

Finally, we prove \cref{prop:near1}, showing that $F$ is the optimal disagreement bound at some points near $1$.
The proof is based on the lower bound established for the total number of disagreements in any $(n,k)$-assignment, and the observation that, when $k\gg n$, the dominant term comes from pairs of disjoint sets (and is identical for all $(n,k)$-assignments) and the next dominant term comes from pairs with intersection of size 1.

\begin{proof}[Proof of \cref{prop:near1}]
  Fix $n \ge 2$ and let $k \gg n$.
  By \cref{prop:optimal-bound-step},
  \[ \sum_{m=1}^n c_m f(\tfrac mn) \ge \sum_{m=1}^n c_m F(\tfrac mn) .\]
  Since $f(1)=F(1)=1$, the term for $m=n$ cancels.
  Since $f\leq 1$ and $F\geq 0$, this leaves
  \[ c_{n-1} f(1-\tfrac1n) \geq c_{n-1} F(1-\tfrac1n) - \sum_{m=1}^{n-2} c_m. \]
  Finally, observe that with $n$ fixed as $k\to\infty$, we have $c_m = \Theta(k^{m-n})$, so $\sum_{m=1}^{n-2} c_m \ll c_{n-1}$ (in other words, the number of distance-$(n-1)$ pairs is much larger than the number of pairs of smaller distance).
  Dividing by $c_{n-1}$ and taking the limit $k\to\infty$, we conclude the proof.
\end{proof}

\section{Open Questions}

As noted, we are unable to show that $F$ is the optimal disagreement bound in the strong sense:

\begin{question}
  Is every disagreement bound pointwise larger-or-equal than $F$?
\end{question}

In light of \cref{T:optimal-bound}, this is equivalent to the following question.

\begin{question}
Let $f_1$ and $f_2$ be two disagreement bounds. Is the pointwise minimum $f_1 \wedge f_2$ also a disagreement bound?
\end{question}

The examples used to give some of the lower bounds above give rise to some questions in extremal combinatorics.
For example, towards bounding $f(\frac23)$, suppose for each set $I\subset\{1,\dots,n\}$ of size $|I|=3$, we assign a number $z_I\in I$.
The number of pairs $(I,J)$ with $|I\cap J|=1$ is $\binom{n}{1,2,2} = \frac{n_5}{4}$, where $n_5 = \frac{n!}{(n-5)!}$.
Among these, consider the number $Q$ of pairs $(I,J)$ such that $z_I=z_J$ is the unique element in $I \cap J$.

\begin{question}
  What is the maximal value of $Q$? 
\end{question}

Taking $z_I = \min(I)$ gives $Q = \frac{n_5}{5}$.
The known bound $f(2/3)\geq F(2/3) = 4/5$ for every disagreement bound $f$ implies that $Q \leq (\frac45 + o(1)) \frac{n_5}{4}$.
A careful modification of $z_I$ only for those sets $I$ with $\min(I) \geq n-4$ gives $Q = \frac{n_5}{5}+4$, which we believe is the maximum possible for every $n$.

Finally, we raise the question of extending our results to multi-marginal optimal transport with general cost functions:

\begin{question}
  Given a cost function $\phi$, for which $f$ does it hold that, for any finite collection of random variables taking values in $\R^d$, there exists a coupling $\mu$ of the variables such that $\E_\mu \phi(X,Y) \le f(d_\phi(X,Y))$ for every two variables $X$ and $Y$ in the collection?
\end{question}

Some results in this direction can be found in the subsequent paper \cite{TA2}.

\subsection*{Acknowledgements.}

OA is supported in part by NSERC.
We are grateful to Russ Lyons for helpful comments on earlier versions of this manuscript, to Noga Alon for the idea for \cref{prop:near1},
to Alessio Figalli and Young-Heon Kim for pointing out the bound \eqref{eq:MMOT} in the context of multi-marginal optimal transport,
and to Oded Regev and the anonymous reviewers for bringing to our attention some of the existing literature and other comments.

\bibliography{coupling}

\end{document}